\documentclass[reqno]{amsart}
\usepackage{amssymb}
\usepackage{graphicx}

\usepackage[usenames, dvipsnames]{color}
\usepackage{verbatim}
\usepackage{mathrsfs}
\usepackage{bm}
\usepackage{cite}

\numberwithin{equation}{section}

\newtheorem{theorem}{Theorem}[section]
\newtheorem{corollary}[theorem]{Corollary}
\newtheorem{lemma}[theorem]{Lemma}

\theoremstyle{definition}
\newtheorem{remark}[theorem]{Remark}

\theoremstyle{definition}

\theoremstyle{definition}

\makeatletter
\def\dashint{\operatorname%
{\,\,\text{\bf-}\kern-.98em\DOTSI\intop\ilimits@\!\!}}
\makeatother

\def\\det{\text{det}}

\def\.5{\frac{1}{2}}

\newcommand{\RN}[1]{%
  \textup{\uppercase\expandafter{\romannumeral#1}}%
}

\renewcommand{\epsilon}{\varepsilon}

\newcounter{marnote}

\begin{document}
\title[Gradient Estimates for Elliptic systems ]{Estimates for Elliptic Systems in a Narrow Region arising from Composite Materials}
\author[H. J. Ju]{Hongjie Ju}
\address[H.J. Ju]{School of  Sciences, Beijing University of Posts  and Telecommunications,
Beijing 100876, China}
\email{hjju@bupt.edu.cn}
\thanks{H.J. Ju   was partially supported by NSFC (11301034) (11471050).}

\author[H.G. Li]{Haigang Li}
\address[H.G. Li]{School of Mathematical Sciences, Beijing Normal University, Laboratory of Mathematics and Complex Systems, Ministry of Education, Beijing 100875, China. }
\email{hgli@bnu.edu.cn}
\thanks{H.G. Li was partially supported by  NSFC (11571042) (11371060) (11631002), Fok Ying Tung Education Foundation (151003). Corresponding author.}

\author[L.J. Xu]{Longjuan Xu}
\address[L.J. Xu]{School of Mathematical Sciences, Beijing Normal University, Laboratory of Mathematics and Complex Systems, Ministry of Education, Beijing 100875, China.}


\date{\today} 



\begin{abstract}
In this paper, we establish the pointwise upper and lower bounds of the gradients of solutions to a class of elliptic systems, including linear systems of elasticity, in a general narrow region and in all dimensions. This problem  arises from the study of damage analysis of high-contrast composite materials. Our results show that the damage may initiate from the narrowest place.
\end{abstract}

\maketitle

\section{Introduction and main results}

From the structure of fiber-reinforced composite, there are a relatively large number of fibers which are touching or nearly touching. The maximal strains can be strongly influenced by the distances between fibers. Especially, in  high-contrast fiber-reinforced composites high concentration of extreme electric field or mechanical loads will occur in the narrow regions between two adjacent fibers. The purpose of this paper is to establish gradient estimates for solutions to a class of elliptic systems, including linear systems of elasticity, in such narrow regions. 

A composite medium would be represented by a bounded domain $\Omega$, divided into a finite number of subdomains. A simple two-dimensional example, which very well illustrates the main feature of our estimates, would have the domain $\Omega\subset\mathbb{R}^{2}$ model the cross-section of a fiber-reinforced composite, with $D_{1}\cup{D}_{2}\subset\Omega$ representing the cross-section of the fibers, the remaining subdomain $\Omega\setminus\overline{D_{1}\cup{D}_{2}}$ representing the matrix surrounding the fibers.
It is well known that for the scalar case, the anti-plane shear model is inconsistent with the two-dimensional conductivity model. Thus, the blow-up analysis for electric field has a valuable meaning in relation to the damage analysis of composite material. The most important quantities from an engineering point of view are $|\nabla{u}|$, representing the electric field in the conductivity problem or the stresses in the anti-plane shear model. Therefore, stimulated by the well-known work on damage analysis of fiber composites \cite{BAPK, K1, M}, there have been a number of papers, starting from \cite{BV,LN,LV}, on gradient estimates for solutions of elliptic equations and systems with piecewise smooth coefficients which are relevant in such studies.
 See \cite{ABTV,ACKLY,ADKL,AKL,AKLLeZ,AKLLiZ,BLY1, BLY2, BLL, BLL2, BT,BT2,KLY,LLBY,LY,Y1,Y2} and the references therein.

In order to investigate the high concentration phenomenon of high-contrast composites when $\mathrm{dist}(D_{1},D_{2})$ is small, it is important to study the gradient estimate for the limiting case of a class of elliptic equations and systems with partially degenerated coefficients, that is, the coefficients in $D_{1}$ and $D_{2}$ degenerate to $\infty$. In a recent paper \cite{ADKL}, some gradient estimates were obtained concerning the conductivity problem where the conductivity is allowed to be $\infty$ (perfect conductor).

{\bf Theorem A (\cite{ADKL}).}  Let $B_1$ and $B_2$ be two balls in $\mathbb{R}^3$ with radius $R$ and centered at $(0, 0, \pm R\pm\frac{\varepsilon}{2})$, respectively, (See Figure \ref{fig:1.1}). Let $H$ be a harmonic function in $\mathbb{R}^3$ such that $H(0)=0$. Define $u$ to be the solutions of
\begin{align*}
\begin{cases}
\Delta u=0,&\hbox{in}\ \mathbb{R}^3\setminus\overline{B_1\cup B_2},\\
u=0, &\hbox{on}\ \partial B_1\cup \partial B_2,\\
u(x)-H(x)=O(|x|^{-1}), &\hbox{as}\ |x|\rightarrow+\infty.
\end{cases}
\end{align*}
Then there exists a constant $C$ independent of $\varepsilon$ such that $$\|\nabla(u-H)\|_{L^\infty(\mathbb{R}^3\setminus\overline{B_1\cup B_2})}\leq C.$$

\begin{figure}[t]
\begin{minipage}[c]{0.43\linewidth}
\centering
\includegraphics[width=1.2in]{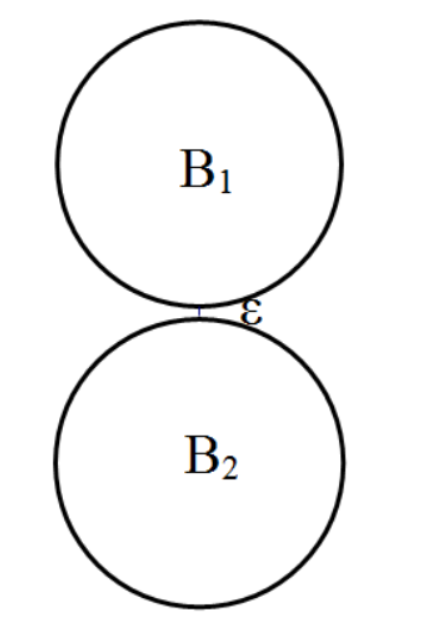}
\caption{\small Two closely spaced balls.}
\label{fig:1.1}
\end{minipage}
\begin{minipage}[c]{0.52\linewidth}
{\centering
\includegraphics[width=2.7in]{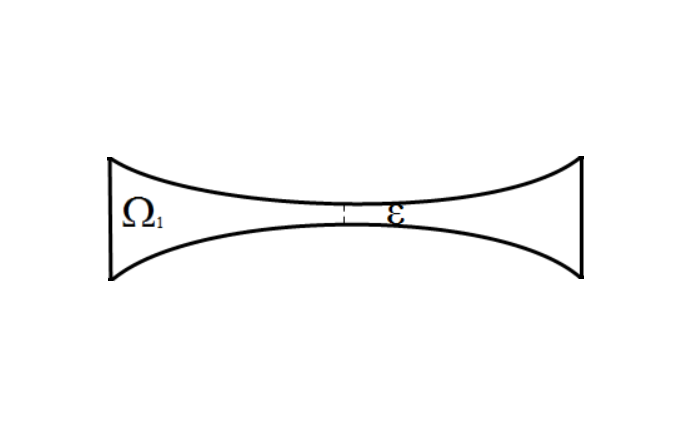}}
\caption{\small A narrow region $\Omega_1$.}
\label{fig:1.2}
\end{minipage}
\end{figure}

Theorem A shows that when the boundary values are both zero on $\partial{B}_{1}$ and $\partial{B}_{2}$, $|\nabla{u}|$ is bounded, so no concentration occurs. Theorem A was extended to the elliptic systems in  \cite{LLBY}, and was improved that   $|\nabla{u}|$ decays exponentially fast near the origin. Later, another proof for scalar case was given in \cite{KLY2}.    However, it is much more interesting to study the case when the boundary data  are different on $\partial{B}_{1}$ and $\partial{B}_{2}$, which more frequently appears in practical engineering applications, see \cite{BC,K1,K2,M}, where it is shown that in dimension two $|\nabla u|$ may blow up in the narrow region between $B_1$ and $B_2$, as $\varepsilon\to 0$.

Contrary to the scalar equation, less is known  on such blow-up  phenomenon
for the linear elasticity case. Therefore, our effort is focussed on the narrow region (see Figure 2) to investigate the gradient estimate for a class of general elliptic systems,   including linear systems of elasticity.

Before stating our results, we first fix our domain. To be precise, we define a more general narrow region in all dimensions as follows: for $r\leq1$,
 $$\Omega_r:=\left\{x=(x',x_{n})\in \mathbb{R}^n~\big|~ -\frac{\varepsilon}{2}+h_2(x')<x_{n}<\frac{\varepsilon}{2}+h_1(x'), ~x'\in B_r(0') \right\},$$
 where $B_r(0'):=\left\{x'=(x_1,\cdots,x_{n-1})\in \mathbb{R}^{n-1} ~\big|~ |x'|<r\right\}$ is a ball in $\mathbb{R}^{n-1}$, centered at the origin $0'$ of radius $r$, $h_1$ and $h_2$ are smooth functions in $B_1(0')$ satisfying
 $$-\frac{\varepsilon}{2}+h_2(x')<\frac{\varepsilon}{2}+h_1(x'),\quad \hbox{for}\ |x'|\leq1,$$
 \begin{align}\label{h0}h_1(0')=h_2(0')=0,\quad \nabla h_1(0')=\nabla h_2(0')=0,\end{align}
 \begin{align}\label{h1}
 \nabla^2 (h_1-h_2)(0')\geq \kappa_0I_{n-1}
 \end{align}and
 \begin{align}\label{eq1.5'}
\|h_1\|_{C^2(B_1(0'))}+\|h_2\|_{C^2(B_1(0'))}\leq \kappa_{1},
\end{align}
 where $I_{n-1}$ is the $(n-1)\times(n-1)$ identity matrix, $\kappa_0$, $ \kappa_1$ are some positive constants.
 We denote the top and bottom boundaries of $\Omega_r$ as
 $$\Gamma_r^+=\{x\in \mathbb{R}^n\, |\, x_{n}=\frac{\varepsilon}{2}+h_1(x'), |x'|\leq r\},~ \Gamma_r^-=\{x\in \mathbb{R}^n \,|\, x_{n}=-\frac{\varepsilon}{2}+h_2(x'), |x'|\leq r\},$$respectively.

Let $u=(u^1,\cdots,u^N)$ be a vector-valued function. Consider the following boundary value problem:
\begin{align}\label{eq1.1}
\begin{cases}
\partial_\alpha\left(A_{ij}^{\alpha\beta}(x)\partial_\beta u^j+B_{ij}^\alpha(x) u^j\right)+C_{ij}^\beta(x) \partial_\beta u^j+D_{ij}(x)u^j=0,\quad&
\hbox{in}\  \Omega_1,  \\
u=\varphi(x),\quad &\hbox{on}\ \Gamma_1^+,\\
u=\psi(x), \quad&\hbox{on} \ \Gamma_1^-,
\end{cases}
\end{align}
where $\varphi(x)=(\varphi^1(x), \varphi^2(x), \cdots, \varphi^N(x))\in C^2(\Gamma_1^+; \mathbb{R}^N), \psi(x)=(\psi^1(x), \psi^2(x), \cdots, $
$\psi^N(x))\in C^2(\Gamma_1^-; \mathbb{R}^N)$ are given vector-valued functions. Here the usual summation convention is used: $\alpha$ and $\beta$ are summed from 1 to $n$, while $i$ and $j$
are summed from 1 to $N$.

The coefficients  $A_{ij}^{\alpha\beta}(x)$ are measurable, bounded, that is,
\begin{align}\label{eq1.2}
|A_{ij}^{\alpha\beta}(x)| \leq \Lambda
\end{align}for some constant $\Lambda>0$
and satisfy the rather weak ellipticity condition, that is,
there exists a constant $0<\lambda<\infty$ such that
\begin{align}\label{eq1.3}
\int_{\Omega_1}A_{ij}^{\alpha\beta}(x)\partial_{\alpha}\xi^i\partial_{\beta}\xi^jdx\geq \lambda\int_{\Omega_1}|\nabla\xi|^2dx,\quad \forall\ \xi\in H_0^1(\Omega_1; \mathbb{R}^N ).
\end{align}
Recall that a system is called a system of elasticity if $N=n$, the coefficients satisfy
$$A_{ij}^{\alpha\beta}(x)=A_{ji}^{\beta\alpha}(x)=A_{\alpha\,j}^{i\beta}(x),$$
and for all $n\times{n}$ symmetric matrices $\xi_{\alpha}^{i}$,
$$\lambda|\xi|\leq\,A_{ij}^{\alpha\beta}(x)\xi_{\alpha}^{i}\xi_{\beta}^{j}\leq\Lambda|\xi|^{2}.$$ It is clear that hypotheses \eqref{eq1.2} and \eqref{eq1.3} are satisfied by the linear systems of elasticity, especially by the Lam\'{e} system,  see \cite{O},$$\lambda\Delta{u}+(\lambda+\mu)\nabla(\nabla\cdot{u})=0.$$  Furthermore, we assume that \begin{align}\label{eq1.4}\|A\|_{C^2(\Omega_1)}+\|B\|_{C^2(\Omega_1)}+\|C\|_{C^2(\Omega_1)}+\|D\|_{C^2(\Omega_1)}\leq \kappa_{2},\end{align}
 for some positive constant $\kappa_{2}$. Throughout the paper, unless otherwise stated, we  use  $C$ to denote some positive constant, whose values may vary from line to line, which depend only on $n,\ N,\ \lambda,\ \Lambda, \ \kappa_{0}, \ \kappa_{1},\ \kappa_{2}$, but not on $\varepsilon$. Also, we call a constant having such dependence a {\it universal constant}.

In the paper, the main result concerns local piecewise gradient estimates of weak solutions $u $ of problem (\ref{eq1.1}); that is, $u\in{H}^{1}(\Omega_1; \mathbb{R}^N)$, and satisfies
$$\int_{\Omega_{1}}\left(A_{ij}^{\alpha\beta}(x)\partial_\beta u^j+B_{ij}^\alpha(x) u^j\right)\partial_{\alpha}\zeta^{i}-C_{ij}^\beta(x) \partial_\beta u^j\zeta^{i}-D_{ij}(x)u^j\zeta^{i}dx=0,$$
for every vector-valued function $\zeta=(\zeta^{1},\cdots,\zeta^{N})\in{C}_{c}^{\infty}(\Omega_{1};\mathbb{R}^N)$, and hence for every $\zeta\in{H}_{0}^{1}(\Omega_1; \mathbb{R}^N)$.

\begin{theorem}\label{thm1.1}
Assume that hypotheses (\ref{h0})--(\ref{eq1.5'}) and (\ref{eq1.2})--(\ref{eq1.4}) are satisfied, and let $u\in H^{1}(\Omega_1; \mathbb{R}^N)$ be a weak solution of problem (\ref{eq1.1}). Then, for $x\in \Omega_{1/2}$,
\begin{align}\label{mainestimate}
|\nabla u(x',x_{n})|
\leq&\, \frac{C}{\varepsilon+|x'|^2}\Big|\varphi(x',\varepsilon/2+h_{1}(x'))-\psi(x',-\varepsilon/2+h_{2}(x'))\Big|\nonumber\\
&+C\left(\|\varphi\|_{C^{2}(\Gamma_{1}^{+})}+\|\psi\|_{C^{2}(\Gamma_{1}^{-})}+\|u\|_{L^2(\Omega_1)}\right).
\end{align}
Moreover, if $\varphi^l(0',\varepsilon/2)\neq\psi^l(0',-\varepsilon/2)$ for some integer $l$, then
\begin{align}\label{lower-bound}|\nabla u(0',x_{n})|\geq\frac{|\varphi^l(0',\varepsilon/2)-\psi^l(0',-\varepsilon/2)|}{C\varepsilon},\quad \forall\ x_n\in\left(-\frac{\varepsilon}{2},\frac{\varepsilon}{2}\right).\end{align}

\end{theorem}

\begin{remark}
Some remarks on Theorem \ref{thm1.1} are in order.

\noindent $(i)$ When $\varphi(x',\varepsilon/2+h_{1}(x'))\equiv\psi(x',-\varepsilon/2+h_{2}(x'))$, we first observe from \eqref{mainestimate} that $|\nabla{u}|\leq\,C$ and there is no blow-up occurring, which is consistent with the main result of \cite{LLBY}.

 \noindent $(ii)$ On the other hand, it is easy to see that if $\varphi=a$, $\psi=b$ for two different constant vectors $a$ and $b$, then $|\nabla{u}|$ will blow up, especially with rate $\varepsilon^{-1}$ at the origin.

\noindent  $(iii)$ When $\varphi(0',\varepsilon/2)=\psi(0',-\varepsilon/2)$, by the Taylor  expansions of $\varphi$ and $\psi$, we have
\begin{align*}
|\nabla u(x',x_{n})|
\leq&\, \frac{C}{\varepsilon+|x'|^2}\Big|\left(\nabla_{x'}\varphi(0',\varepsilon/2)-\nabla_{x'}\psi(0',-\varepsilon/2)\right)\cdot x'+O(|x'|^2)\Big|\nonumber\\
&+C\left(\|\varphi\|_{C^{2}(\Gamma_{1}^{+})}+\|\psi\|_{C^{2}(\Gamma_{1}^{-})}+\|u\|_{L^2(\Omega_1)}\right),\ x\in\Omega_{1/2}.
\end{align*}
Clearly, if $\nabla_{x'}\varphi(0',\varepsilon/2)\neq\nabla_{x'}\psi(0',-\varepsilon/2)$,  then $|\nabla{u}|\leq C$  on $\Omega_\varepsilon$, where $C$ is independent of $\varepsilon$.  If $\nabla_{x'}\varphi(0',\varepsilon/2)=\nabla_{x'}\psi(0',-\varepsilon/2)$,  $|\nabla{u}|$ is uniformly bounded on $\Omega_{1/2}$. Consequently, in the case that $\varphi(0',\varepsilon/2)=\psi(0',-\varepsilon/2)$, there is no blow-up occurring at the origin.

Theorem \ref{thm1.1} gives more information about the dependence of $|\nabla{u}|$, which will play an important role in the study of the perfect conductivity problem (e.g. \cite{BLY1,BLY2,LX}) and Lam\'e system with partially infinite coefficients (e.g. \cite{BLL,BLL2,BJL}), where the coefficients in the inclusions are allowed to be $\infty$.
\end{remark}
For the convenience of further applications, we list the analog result for the conductivity problem in the narrow region as a consequence. For the boundary value problem of Laplace equation
\begin{align}\label{eq1.1'}
\begin{cases}
\Delta u=0,\quad&
\hbox{in}\  \Omega_1,  \\
u=\varphi(x),\quad &\hbox{on}\ \Gamma_1^+,\\
u=\psi(x), \quad&\hbox{on} \ \Gamma_1^-,
\end{cases}
\end{align}we have
\begin{corollary}\label{corol1.2}
 Assume that $u\in H^{1}(\Omega_1)$ is a weak solution of (\ref{eq1.1'}), $\varphi(x)\in C^2(\Gamma_1^+), \ \psi(x)\in C^2(\Gamma_1^-)$ are given functions. Then, for $x\in \Omega_{1/2}$,
\begin{align}\label{mainestimate'}
|\nabla u(x',x_{n})|
\leq&\, \frac{C}{\varepsilon+|x'|^2}\Big|\varphi(x',\varepsilon/2+h_{1}(x'))-\psi(x',-\varepsilon/2+h_{2}(x'))\Big|\nonumber\\
&+C\left(\|\varphi\|_{C^{2}(\Gamma_{1}^{+})}+\|\psi\|_{C^{2}(\Gamma_{1}^{-})}+\|u\|_{L^2(\Omega_1)}\right).
\end{align}
If $\varphi(0',\varepsilon/2)\neq\psi(0',-\varepsilon/2)$, then
\begin{align*}|\nabla u(0',x_{n})|\geq\frac{|\varphi(0',\varepsilon/2)-\psi(0',-\varepsilon/2)|}{C\varepsilon},\quad \forall\ x_n\in\left(-\frac{\varepsilon}{2}, \frac{\varepsilon}{2}\right).\end{align*}

\end{corollary}

Remark 1.2 is also true for the problem (\ref{eq1.1'}).

The paper is organized as follows. In Section 2, we use energy method and an adaptive version of  Bao-Li-Li's iteration technique \cite{BLL} to prove Theorem \ref{thm1.1}. The main differences of the proof of Corollary 1.3 with that of Theorem 1.1 are given   in Section 3.

\vspace{.5cm}

\section{Proof of Theorem 1.1}\label{sec_2}

We decompose the solution of (\ref{eq1.1}) as follows:
\begin{align*}
u=v_1+v_2+\cdots+v_N,
\end{align*}
where $v_l=(v_l^1, v_l^2, \cdots, v_l^N)$, $l=1,2, \cdots, N$, with  $v_l^j=0$ for $j\neq l$, and $v_l$ satisfies the following boundary value problem
\begin{align}\label{eq_v2.1}
\begin{cases}
  \partial_\alpha\left(A_{ij}^{\alpha\beta}(x)\partial_\beta v_l^j+B_{ij}^\alpha(x) v_l^j\right)+C_{ij}^\beta(x) \partial_\beta v_l^j+D_{ij}(x)v_l^j=0,
&\hbox{in}\  \Omega_1,  \\
v_l=(0,\cdots, 0, \varphi^l, 0,\ \cdots, 0),\ &\hbox{on}\ \Gamma_1^+,\\
v_l=(0,\cdots, 0, \psi^l, 0,\ \cdots, 0),&\hbox{on} \ \Gamma_1^-.
\end{cases}
\end{align}
Then
\begin{equation*}
\nabla{u}=\sum_{l=1}^{N}\nabla{v}_{l}.
\end{equation*}
In order to estimate $|\nabla v_l|$, $l=1,\cdots,N$, we  introduce a scalar function $\bar{u}\in C^2(\mathbb{R}^n)$ such that $\bar{u}=1$ on $\Gamma_1^+$, $\bar{u}=0$ on $\Gamma_1^-$ and
\begin{align}\label{bar_u}\bar{u}(x)=\frac{x_{n}-h_2(x')+\frac{\varepsilon}{2}}{\varepsilon+h_1(x')-h_2(x')},\quad\hbox{in}\ \Omega_1.\end{align} By a direct calculation, we obtain that
\begin{align}
|\partial_{\alpha}\bar{u}(x)|\leq\frac{C|x'|}{\varepsilon+|x'|^2},~ \alpha=1,\cdots, n-1,~~\frac{1}{C(\varepsilon+|x'|^2)}\leq|\partial_{n}\bar{u}(x)|\leq\frac{C}{\varepsilon+|x'|^2},\label{e2.4}
\end{align}
and for $\alpha,\ \beta=1,\cdots, n-1,$
\begin{align}
|\partial_{\alpha \beta}\bar{u}(x)|\leq\frac{C}{\varepsilon+|x'|^2},\quad \quad|\partial_{\alpha n}\bar{u}(x)|\leq\frac{C|x'|}{(\varepsilon+|x'|^2)^2},\quad \partial_{nn}\bar{u}(x)=0.\label{ee2.4}
\end{align}

For $l=1,2,\cdots, N$, we define
\begin{equation}\label{equ_tildeu}
\tilde{u}_{l}(x)=(0, \cdots, 0, \varphi^{l}(x',\varepsilon/2+h_{1}(x'))\bar{u}(x)+\psi^{l}(x',-\varepsilon/2+h_{2}(x'))(1-\bar{u}(x)), 0, \cdots, 0).
\end{equation}
 Thus, in view of (\ref{e2.4}) and (\ref{ee2.4}),
\begin{align}
&|\nabla_{x'}\tilde{u}_l(x)|\leq\frac{C|x'|}{\varepsilon+|x'|^2}|\varphi^l(x',\varepsilon/2+h_{1}(x'))-\psi^l(x',-\varepsilon/2+h_{2}(x'))|\nonumber\\
&\ \quad\quad\quad\quad\quad+C(\|\nabla\varphi^l\|_{L^{\infty}}+\|\nabla\psi^l\|_{L^{\infty}}),\label{eq1.7}\\
&\frac{|\varphi^l(x',\varepsilon/2+h_{1}(x'))-\psi^l(x',-\varepsilon/2+h_{2}(x'))|}{C(\varepsilon+|x'|^2)}\leq
\nonumber\\
&|\partial_{n}\tilde{u}_l(x)|\leq\frac{C|\varphi^l(x',\varepsilon/2+h_{1}(x'))
-\psi^l(x',-\varepsilon/2+h_{2}(x'))|}{\varepsilon+|x'|^2},\label{eq1.7a}
\end{align}
and by using (\ref{e2.4}) and (\ref{ee2.4}), for $\alpha,\ \beta=1,\cdots, n-1$,
\begin{align}
|\partial_{\alpha \beta}\tilde{u}_l(x)|&\leq\frac{C}{\varepsilon+|x'|^2}|\varphi^l(x',\varepsilon/2+h_{1}(x'))-\psi^l(x',-\varepsilon/2+h_{2}(x'))|\nonumber\\
&\quad+C\left(\frac{|x'|}{\varepsilon+|x'|^2}+1\right)(\|\nabla\varphi^l\|_{L^{\infty}}+
\|\nabla\psi^l\|_{L^{\infty}})\nonumber\\
&\quad+C(\|\nabla^{2}\varphi^l\|_{L^{\infty}}+\|\nabla^{2}\psi^l\|_{L^{\infty}}),\label{eq1.8}\\
|\partial_{\alpha n}\tilde{u}_{l}(x)|
&\leq\frac{C|x'|}{(\varepsilon+|x'|^2)^2}|\varphi^{l}(x',\varepsilon/2+h_{1}(x'))-\psi(x',-\varepsilon/2+h_{2}(x'))|\nonumber\\
&\quad+\frac{C}{\varepsilon+|x'|^2}(\|\nabla\varphi^{l}\|_{L^{\infty}}+\|\nabla\psi^{l}\|_{L^{\infty}}),\label{eq1.9}\\
\partial_{n n}\tilde{u}_{l}(x)&=0.\label{eq1.10}
\end{align}
Here and throughout the paper, for simplicity we use $\|\nabla\varphi\|_{L^{\infty}}$,  $\|\nabla\psi\|_{L^{\infty}}$, $\|\nabla^{2}\varphi\|_{L^{\infty}}$ and $\|\nabla^{2}\psi\|_{L^{\infty}}$ to denote $\|\nabla\varphi\|_{L^{\infty}(\Gamma_{1}^{+})}$,  $\|\nabla\psi\|_{L^{\infty}(\Gamma_{1}^{-})}$, $\|\nabla^{2}\varphi\|_{L^{\infty}(\Gamma_{1}^{+})}$ and $\|\nabla^{2}\psi\|_{L^{\infty}(\Gamma_{1}^{-})}$,  respectively.

Let $$w_l=v_l-\tilde{u}_l,\qquad l=1,\cdots,N.$$
Then $w$ satisfies
\begin{align}\label{eq2.6}
\begin{cases}
  \partial_\alpha\left(A_{ij}^{\alpha\beta}(x)\partial_\beta w^j+B_{ij}^\alpha(x) w^j\right)+C_{ij}^\beta(x) \partial_\beta w^j+D_{ij}(x)w^j=\tilde{f}^{i},&
\hbox{in}\  \Omega_1,  \\
w=0, \quad&\hbox{on} \ \Gamma_1^\pm,
\end{cases}
\end{align}
where
\begin{align*}
 \tilde{f}^{i}=&
-\partial_\alpha \left(A_{ij}^{\alpha\beta}(x)\partial_\beta \tilde{u}^j+B_{ij}^\alpha(x) \tilde{u}^j+C_{ij}^\alpha(x) \tilde{u}^j\right)\\
&+\partial_\beta (C_{ij}^\beta(x))  \tilde{u}^j-D_{ij}(x)\tilde{u}^j.\nonumber
\end{align*}
Let $\tilde{f}:=(\tilde{f}^{1}, \cdots, \tilde{f}^{N})$,
then it follows from (\ref{eq1.4}) and \eqref{equ_tildeu}--(\ref{eq1.10}) that for $(x',x_n)\in \Omega_1$,
\begin{align}\label{f}
|\tilde{f}(x',x_n)|\leq&\, C|\nabla^2\tilde{u}(x',x_n)|+C|\nabla\tilde{u}(x',x_n)|+C|\tilde{u}(x',x_n)|\nonumber\\
\leq&\left(\frac{C}{\varepsilon+|x'|^2}+\frac{C|x'|}{(\varepsilon+|x'|^2)^2}\right)
|\varphi(x',\varepsilon/2+h_{1}(x'))-\psi(x',-\varepsilon/2+h_{2}(x'))|\nonumber\\
&+\left(\frac{C}{\varepsilon+|x'|^2}+\frac{C|x'|}{\varepsilon+|x'|^2}\right)
(\|\nabla\varphi\|_{L^{\infty}}+\|\nabla\psi\|_{L^{\infty}})\nonumber\\
&+C(\|\nabla^{2}\varphi\|_{L^{\infty}}
+\|\nabla^{2}\psi\|_{L^{\infty}}),
\end{align}
where $C$ is independent of $\varepsilon$.

\begin{lemma}\label{lem2.1}
Let $v_l\in H^1(\Omega_1; \mathbb{R}^N)$ be a weak solution of (\ref{eq_v2.1}), then
\begin{align}\label{lem2.2equ}
\int_{\Omega_{1/2}}|\nabla w_l|^2dx\leq C\left(\|w_l\|^2_{L^2(\Omega_1)}+\|\varphi^{l}\|_{C^{2}(\Gamma_{1}^{+})}^{2}
+\|\psi^{l}\|_{C^{2}(\Gamma_{1}^{-})}^{2}\right),\qquad\,l=1,\cdots,N,
\end{align}
where $C$ depends on $n$, $\lambda$, $\kappa_0$, $\kappa_1$ and $\kappa_2$.
\end{lemma}

\begin{proof} For simplicity, we assume that $\psi\equiv0$.
We only prove the case when $l=1$ for instance. The other cases are the same. Denote
$$w:=w_1, \ \tilde{u}:=\tilde{u}_1 \ \mbox{and}\quad \varphi:=\varphi^1.$$
Then,
\begin{align*}
 \tilde{f}^{i}
=&-\partial_\alpha \left(A_{i1}^{\alpha\beta}(x)\partial_\beta \tilde{u}^1+B_{i1}^\alpha(x) \tilde{u}^1+C_{i1}^\alpha(x) \tilde{u}^1\right)\nonumber\\
&+\partial_\beta (C_{i1}^\beta(x))\tilde{u}^1-D_{i1}(x)\tilde{u}^1,\nonumber
\end{align*}
and it follows from (\ref{eq1.4}) and \eqref{equ_tildeu}--(\ref{eq1.10}) that
\begin{align}\label{fg}
|\tilde{f}^{i}(x)|
&\leq\,C(|\nabla^2\tilde{u}^1(x)|+|\nabla\tilde{u}^1(x)|+|\tilde{u}^1(x)|)\nonumber\\
&\leq C\|\varphi\|_{C^2(\Gamma_1^+)},\quad\,x\in\Omega_{1}\setminus\overline{\Omega_{1/4}}.
\end{align}
Multiplying the equation in (\ref{eq2.6}) by $w$ and applying integration by parts in $\Omega_{1/2}$, we have
\begin{align*}
&\int_{\Omega_{1/2}}A_{ij}^{\alpha\beta}(x)\partial_\beta w^j\partial_\alpha w^idx
\\
=&-\int_{\Omega_{1/2}}B_{ij}^{\alpha}(x) w^j\partial_\alpha w^idx
+\int_{\Omega_{1/2}}C_{ij}^{\beta}(x)\partial_\beta w^j w^idx+\int_{\Omega_{1/2}}D_{ij}(x) w^jw^idx\\
&-\int_{\Omega_{1/2}}\tilde{f}^{i}w^idx+\int\limits_{\scriptstyle |x'|={\frac{1}{2}},\atop\scriptstyle
-\frac{\varepsilon}{2}+h_2(x')<x_{n}<\frac{\varepsilon}{2}+h_1(x')\hfill}\left(A_{ij}^{\alpha\beta}(x)\partial_\beta w^j+B_{ij}^\alpha(x) w^j\right)w^{i}\frac{x_{\alpha}}{r}ds.
\end{align*}
Using the weak ellipticity condition and the Cauchy inequality, we obtain
\begin{align}\label{equ1}
\lambda\int_{\Omega_{1/2}}|\nabla w|^2dx
\leq&\,\int_{\Omega_{1/2}}A_{ij}^{\alpha\beta}(x)\partial_\beta w^j\partial_\alpha w^idx\nonumber\\
\leq&\,\frac{\lambda}{4}\int_{\Omega_{1/2}}|\nabla w|^2dx+C \int_{\Omega_{1/2}}|w|^2dx+\left|\int_{\Omega_{1/2}}\tilde{f}^{i}w^idx\right|\nonumber\\
&+C\int\limits_{\scriptstyle |x'|={\frac{1}{2}},\atop\scriptstyle
-\frac{\varepsilon}{2}+h_2(x')<x_{n}<\frac{\varepsilon}{2}+h_1(x')\hfill}\left(|\nabla w|^2+|w|^2\right)ds,
\end{align}
Note that $w=0$ on $\Gamma_1^{\pm}$ and $\overline{\Omega_{2/3}}\setminus \Omega_{1/3}\subset\left((\Omega_{1}\setminus\overline{\Omega_{1/4}})\cup (\Gamma_1^{\pm}\setminus \Gamma_{1/4}^{\pm})\right)$,
 by using  the Sobolev embedding theorem and classical $W^{2, p}$ estimates for elliptic systems, 
 we have, for some $p>n$,
\begin{align*}
\|\nabla w\|_{L^\infty(\Omega_{2/3}\setminus \overline{\Omega_{1/3}})}&\leq C\|w\|_{W^{2,p}(\Omega_{2/3}\setminus \overline{\Omega_{1/3}})}\\
&\leq C\left(\|w\|_{L^2(\Omega_{1}\setminus \overline{\Omega_{1/4}})}+\|\tilde{f}\|_{L^\infty(\Omega_{1}\setminus \overline{\Omega_{1/4}})}\right)\nonumber\\
&\leq C\left(\|w\|_{L^2(\Omega_{1})}+\|\varphi\|_{C^2(\Gamma_1^{+})}\right),
\end{align*}
and for $x=(x',x_n)\in\Omega_{2/3}\setminus \overline{\Omega_{1/3}}$,
\begin{align*}
|w(x',x_n)|&=|w(x',x_n)-w(x',\frac{\varepsilon}{2}+h_1(x'))|\\&\leq C(\varepsilon+|x'|^2)\|\nabla w\|_{L^\infty(\Omega_{2/3}\setminus  \overline{\Omega_{1/3}})}\\
&\leq C\left(\|w\|_{L^2(\Omega_{1})}+\|\varphi\|_{C^2(\Gamma_1^{+})}\right).
\end{align*}
In particular, this implies that
\begin{align}\label{w_on|x'|=r}
\int\limits_{\scriptstyle |x'|={\frac{1}{2}},\atop\scriptstyle
-\frac{\varepsilon}{2}+h_2(x')<x_{n}<\frac{\varepsilon}{2}+h_1(x')\hfill}\left(| w|^{2}+|\nabla w|^{2}\right)ds\leq\,C\left[\|w\|^2_{L^2(\Omega_1)}+\|\varphi\|^2_{C^2(\Gamma_1^+)}\right],
\end{align}where $C$ depends only on $n,\ \lambda$ and $\kappa_0$.

Obviously, 
\begin{align}\label{nablax'tildeu on |x'|=r}
\int\limits_{\scriptstyle |x'|={\frac{1}{2}},\atop\scriptstyle
-\frac{\varepsilon}{2}+h_2(x')<x_{n}<\frac{\varepsilon}{2}+h_1(x')\hfill}|\nabla_{x'}\tilde{u}|^2ds\leq C\|\varphi\|^2_{C^1(\Gamma_1^+)},\quad\quad
\end{align}
 and 
\begin{align}\label{nablax'tildeu}
&\int_{\Omega_{1/2}}|\nabla_{x'}\tilde{u}|^2dx\nonumber\\
&\leq C\int_{|x'|<{\frac{1}{2}}}(\varepsilon+h_1(x')-h_2(x'))\left(\frac{|x'|^2|\varphi|^{2}}{(\varepsilon+|x'|^2)^2}+\|\nabla\varphi\|_{L^{\infty}}^{2}\right)dx'\nonumber\\
&\leq C\|\varphi\|_{C^{1}(\Gamma_{1}^{+})}^{2}.
\end{align}
 Applying integration by parts and making use of \eqref{eq1.10} and  \eqref{w_on|x'|=r}--\eqref{nablax'tildeu}, we have
\begin{align}
&\left|\int_{\Omega_{1/2}}\tilde{f}^{i}w^idx\right|\nonumber\\
\leq&\, \left|\int_{\Omega_{1/2}}\sum_{\alpha+\beta<2n}A_{i1}^{\alpha\beta}w^i\partial_{\alpha \beta}\tilde{u}^1dx\right|+\left|\int_{\Omega_{1/2}}\partial_{\alpha}A_{i1}^{\alpha\beta}w^i\partial_{\beta}\tilde{u}^1dx\right|
 +\left|\int_{\Omega_{1/2}}(B_{i1}^\alpha+C_{i1}^\alpha)\tilde{u}^1\partial_{\alpha}w^idx\right|\nonumber\\
 &+\left|\int_{\Omega_{1/2}}(\partial_\beta (C_{i1}^\beta(x))-D_{i1}(x))\tilde{u}^1w^idx\right|+\left|\int\limits_{\scriptstyle |x'|={\frac{1}{2}},\atop\scriptstyle
-\frac{\varepsilon}{2}+h_2(x')<x_{n}<\frac{\varepsilon}{2}+h_1(x')\hfill}(B_{i1}^\alpha+C_{i1}^\alpha)\tilde{u}^1w^i\frac{x_\alpha}{r}ds\right|\nonumber
\end{align}
\begin{align}\label{w2}
\leq& C\int_{\Omega_{1/2}}|\nabla_{x'}\tilde{u}||\nabla w|dx+C\int_{\Omega_{1/2}}|\nabla_{x'}\tilde{u}||w|dx
 +C\int_{\Omega_{1/2}}|\tilde{u}||\nabla w|dx\nonumber\\
 &+C\int_{\Omega_{1/2}}|\tilde{u}||w|dx
 +C\int\limits_{\scriptstyle |x'|={\frac{1}{2}},\atop\scriptstyle
-\frac{\varepsilon}{2}+h_2(x')<x_{n}<\frac{\varepsilon}{2}+h_1(x')\hfill}\left(|\nabla_{x'}\tilde{u}||w|+|\tilde{u}||w|\right)ds\nonumber\\
\leq&\, C\left(\int_{\Omega_{1/2}}|\nabla_{x'}\tilde{u}|^2dx\right)^{\frac{1}{2}}\left(\int_{\Omega_{1/2}}|\nabla w|^2dx\right)^{\frac{1}{2}}+C\left(\int_{\Omega_{1/2}}|\nabla_{x'}\tilde{u}|^2dx\right)^{\frac{1}{2}}\left(\int_{\Omega_{1/2}}|w|^2dx\right)^{\frac{1}{2}}\nonumber\\
&+C\left(\int_{\Omega_{1/2}}|\tilde{u}|^2dx\right)^{\frac{1}{2}}\left(\int_{\Omega_{1/2}}|\nabla w|^2dx\right)^{\frac{1}{2}}
+C\left(\int_{\Omega_{1/2}}|\tilde{u}|^2dx\right)^{\frac{1}{2}}\left(\int_{\Omega_{1/2}}|w|^2dx\right)^{\frac{1}{2}}\nonumber\\
&+C\int\limits_{\scriptstyle |x'|={\frac{1}{2}},\atop\scriptstyle
-\frac{\varepsilon}{2}+h_2(x')<x_{n}<\frac{\varepsilon}{2}+h_1(x')\hfill}(|\nabla_{x'}\tilde{u}|^2+|u|^2)ds+C\int\limits_{\scriptstyle |x'|={\frac{1}{2}},\atop\scriptstyle
-\frac{\varepsilon}{2}+h_2(x')<x_{n}<\frac{\varepsilon}{2}+h_1(x')\hfill}|w|^2ds\nonumber\\
\leq&\, C \|\varphi\|_{C^{1}(\Gamma_{1}^{+})} \left(\int_{\Omega_{1/2}}|\nabla w|^2dx\right)^{\frac{1}{2}}+C(\|w\|^2_{L^2(\Omega_1)}+\|\varphi\|_{C^{2}(\Gamma_{1}^{+})}^{2})\quad\quad\quad\quad\quad\quad\nonumber\\
\leq&\,\frac{\lambda}{4}\int_{\Omega_{1/2}}|\nabla w|^2dx+C(\|w\|^2_{L^2(\Omega_1)}+\|\varphi\|_{C^{2}(\Gamma_{1}^{+})}^{2}).
\end{align}
Inserting  (\ref{w_on|x'|=r}) and (\ref{w2}) to (\ref{equ1}), we obtain that
\begin{align*}
\int_{\Omega_{1/2}}|\nabla w|^2dx\leq C\left(\|w\|^2_{L^2(\Omega_1)}+\|\varphi\|_{C^{2}(\Gamma_{1}^{+})}^{2}\right).
\end{align*}
 Lemma \ref{lem2.1} is established.
\end{proof}

Denote $$\delta(x'):=\varepsilon+h_1(x')-h_2(x').$$  By (\ref{h1}) and (\ref{eq1.5'}), we have
\begin{align}\label{deltabdd}
\frac{1}{C}(\varepsilon+|x'|^2)\leq\delta(x')\leq C(\varepsilon+|x'|^2).
\end{align}
For $x_{0}\in\Omega_{1/2}$, we set
\begin{align}\label{omega_s}\widehat{\Omega}_s(x_{0}):=\left\{~x\in\Omega_{1/2} ~\big|~ |x'-x_{0}'|<s~\right\}, \quad\forall~ 0\leq{s}\leq1/2.\end{align}

\begin{lemma}\label{lem2.2}
For $0\leq|x_{0}'|\leq\sqrt{\varepsilon}$,
\begin{align}\label{step2}
 \int_{\widehat{\Omega}_\delta(x_{0})}|\nabla w_{l}|^2dx
 &\leq C\varepsilon^{n-1}[|\varphi^{l}(x_{0}',\varepsilon/2+h_{1}(x_{0}'))
 -\psi^{l}(x_{0}',-\varepsilon/2+h_{2}(x_{0}'))|^2\nonumber\\
 &\quad+\varepsilon(\|\varphi^{l}\|_{C^2(\Gamma_{1}^{+})}^2+\|\psi^{l}\|_{C^2(\Gamma_{1}^{-})}^2+\|w_l\|^2_{L^2(\Omega_1)})];
\end{align}
and for $\sqrt{\varepsilon}<|x_{0}'|<\frac{1}{2}$,
\begin{align}\label{Step2}
 \int_{\widehat{\Omega}_\delta(x_{0})}|\nabla w_{l}|^2dx
 &\leq C|x_{0}'|^{2(n-1)}[|\varphi^{l}(x_{0}',\varepsilon/2+h_{1}(x_{0}'))-\psi^{l}(x_{0}',-\varepsilon/2+h_{2}(x_{0}'))|^2\nonumber\\
 &\quad+|x_{0}'|^{2}(\|\varphi^{l}\|_{C^2(\Gamma_{1}^{+})}^2+\|\psi^{l}\|_{C^2(\Gamma_{1}^{-})}^2+\|w_l\|^2_{L^2(\Omega_1)})],
\end{align}
where $\delta=\delta(x_{0}')$, $l=1,\cdots,N$.
\end{lemma}

\begin{proof}
We still assume that $\psi\equiv0$ and only prove the case when $l=1$ for instance, and denote $w:=w_1$, $\tilde{u}:=\tilde{u}_{1}$ and $\varphi:=\varphi^1$.

For $0<t<s<1$,  let $\eta(x')$ be a smooth function satisfying $0\leq \eta(x')\leq1$, $\eta(x')=1$ if $|x'-x_{0}'|<t$, $\eta(x')=0$ if $|x'-x_{0}'|>s$ and $|\nabla\eta(x')|\leq\frac{2}{s-t}$. Multiplying $\eta^2w$ on both sides of the equation in (\ref{eq2.6}) and applying integration by parts, we have
\begin{align*}
&-\int_{\widehat{\Omega}_s(x_{0})}(A_{ij}^{\alpha\beta}(x)\partial_\beta w^j+B_{ij}^{\alpha}(x)w^j)\partial_\alpha(\eta^2w^i)dx
\nonumber\\& +\int_{\widehat{\Omega}_s(x_{0})}C_{ij}^{\beta}(x)\partial_\beta w^j\eta^2w^idx+\int_{\widehat{\Omega}_s(x_{0})}D_{ij}(x) w^j\eta^2w^idx
=\int_{\widehat{\Omega}_s(x_{0})}\tilde{f}^{i}\eta^2w^idx.
\end{align*}
Since\begin{align*}&\int_{\widehat{\Omega}_s(x_{0})}(A_{ij}^{\alpha\beta}(x)\partial_\beta w^j+B_{ij}^{\alpha}(x)w^j)\partial_\alpha(\eta^2w^i)dx\\
=&\,\int_{\widehat{\Omega}_s(x_{0})}A_{ij}^{\alpha\beta}(x)\partial_\beta(\eta w^j)\partial_\alpha(\eta w^i)dx
-\int_{\widehat{\Omega}_s(x_{0})}A_{ij}^{\alpha\beta}(x)(\partial_\beta\eta w^j)\partial_\alpha(\eta w^i)dx\\
&+\int_{\widehat{\Omega}_s(x_{0})}A_{ij}^{\alpha\beta}(x)\partial_\beta(\eta w^j)(\partial_\alpha\eta w^i)dx-\int_{\widehat{\Omega}_s(x_{0})}A_{ij}^{\alpha\beta}(x)(\partial_\beta\eta w^j)(\partial_\alpha\eta w^i)dx,\\
&+\int_{\widehat{\Omega}_s(x_{0})}B_{ij}^{\alpha}(x)(\eta w^j)\partial_\alpha(\eta w^i)dx+\int_{\widehat{\Omega}_s(x_{0})}B_{ij}^{\alpha}(x)(\partial_\alpha \eta w^j)(\eta w^i)dx,
\end{align*}
and
\begin{align*}
&\int_{\widehat{\Omega}_s(x_{0})}C_{ij}^{\beta}(x)\partial_\beta w^j\eta^2w^idx\\
=&\,\int_{\widehat{\Omega}_s(x_{0})}C_{ij}^{\beta}(x)\partial_\beta(\eta w^j)(\eta w^i)dx-\int_{\widehat{\Omega}_s(x_{0})}C_{ij}^{\beta}(x)(\partial_\beta \eta w^j)(\eta w^i)dx,
\end{align*}
by using the weak ellipticity condition (\ref{eq1.3}) and the Cauchy inequality, we have
\begin{align*}&\lambda\int_{\widehat{\Omega}_s(x_{0})}|\nabla(\eta w)|^2dx\leq\int_{\widehat{\Omega}_s(x_{0})}A_{ij}^{\alpha\beta}(x)\partial_\beta(\eta w^j)\partial_\alpha(\eta w^i)dx\\
=&\,\int_{\widehat{\Omega}_s(x_{0})}A_{ij}^{\alpha\beta}(x)(\partial_\beta\eta w^j)\partial_\alpha(\eta w^i)dx
-\int_{\widehat{\Omega}_s(x_{0})}A_{ij}^{\alpha\beta}(x)\partial_\beta(\eta w^j)(\partial_\alpha\eta w^i)dx\\
&+\int_{\widehat{\Omega}_s(x_{0})}A_{ij}^{\alpha\beta}(x)(\partial_\beta\eta w^j)(\partial_\alpha\eta w^i)dx
-\int_{\widehat{\Omega}_s(x_{0})}B_{ij}^{\alpha}(x)(\eta w^j)\partial_\alpha(\eta w^i)dx\\
&-\int_{\widehat{\Omega}_s(x_{0})}B_{ij}^{\alpha}(x)(\partial_\alpha \eta w^j)(\eta w^i)dx+\int_{\widehat{\Omega}_s(x_{0})}C_{ij}^{\beta}(x)\partial_\beta(\eta w^j)(\eta w^i)dx\\
&-\int_{\widehat{\Omega}_s(x_{0})}C_{ij}^{\beta}(x)(\partial_\beta \eta w^j)(\eta w^i)dx\nonumber +\int_{\widehat{\Omega}_s(x_{0})}D_{ij}(x) (\eta w^j)(\eta w^i)dx-\int_{\widehat{\Omega}_s(x_{0})}\eta^{2}\tilde{f}^{i}w^idx\\
\leq&\, \frac{\lambda}{2}\int_{\widehat{\Omega}_s(x_{0})}|\nabla(\eta w)|^2dx+C\int_{\widehat{\Omega}_s(x_{0})}|(\nabla\eta) w|^2dx+\frac{C}{(s-t)^2}\int_{\widehat{\Omega}_s(x_0)}|w|^2dx \\
&+(s-t)^2\int_{\widehat{\Omega}_s(x_0)}|\tilde{f}|^2dx\\
\leq&\, \frac{\lambda}{2}\int_{\widehat{\Omega}_s(x_{0})}|\nabla(\eta w)|^2dx+\frac{C}{(s-t)^2}\int_{\widehat{\Omega}_s(x_{0})}|w|^2dx+(s-t)^2\int_{\widehat{\Omega}_s(x_{0})}|\tilde{f}|^2dx.
\end{align*}
Thus, we obtain
\begin{align}\label{ww}
\int_{\widehat{\Omega}_t(x_{0})}|\nabla w|^2dx\leq\frac{C}{(s-t)^2}\int_{\widehat{\Omega}_s(x_{0})}|w|^2dx +C(s-t)^2\int_{\widehat{\Omega}_s(x_{0})}|\tilde{f}|^2dx.
\end{align}
Note that $w=0$ on $\Gamma_1^-$, by (\ref{eq1.5'}) and the H\"{o}lder inequality, we obtain
\begin{align}\label{w}
\int_{\widehat{\Omega}_s(x_{0})}|w|^2dx&=\int_{\widehat{\Omega}_s(x_{0})}\left|\int_{-\frac{\varepsilon}{2}+h_2(x')}^{x_{n}}\partial_nw(x', x_{n})dx_{n}\right|^2dx\nonumber\\
&\leq \int_{\widehat{\Omega}_s(x_{0})}(\varepsilon+h_1(x')-h_2(x'))\int_{-\frac{\varepsilon}{2}+h_2(x')}^{\frac{\varepsilon}{2}+h_1(x')}|\nabla w|^2dx_{n}dx\nonumber\\
&\leq\int_{|x'-x_{0}'|<s}C(\varepsilon+|x'|^2)^2\int_{-\frac{\varepsilon}{2}+h_2(x')}^{\frac{\varepsilon}{2}+h_1(x')}|\nabla w|^2dx_{n}dx'.
\end{align}
It follows from (\ref{f}) and the mean value theorem that
\begin{align}
&\int_{\widehat{\Omega}_s(x_{0})}|\tilde{f}|^2dx\nonumber\\
\leq&\, |\varphi(x_{0}',\varepsilon/2+h_{1}(x_{0}'))|^2\int_{\widehat{\Omega}_s(x_{0})}\left(\frac{C}{\varepsilon+|x'|^2}+\frac{C|x'|}{(\varepsilon+|x'|^2)^2}\right)^2dx\nonumber\\
&+\|\nabla\varphi\|_{L^\infty}^2\int_{\widehat{\Omega}_s(x_{0})}\left(\frac{C}{\varepsilon+|x'|^2}+\frac{C|x'|}{(\varepsilon+|x'|^2)^2}\right)^2|x'-x_0'|^2dx\nonumber\\
&+\|\nabla\varphi\|_{L^\infty}^2\int_{\widehat{\Omega}_s(x_{0})}\left(\frac{C}{\varepsilon+|x'|^2}+\frac{C|x'|}{\varepsilon+|x'|^2}\right)^2dx+Cs^{n-1}\|\nabla^2\varphi\|_{L^\infty}^2\nonumber\\
\leq&\, C|\varphi(x_{0}',\varepsilon/2+h_{1}(x_{0}'))|^2 \int_{|x'-x_{0}'|<s}\frac{1}{(\varepsilon+|x'|^2)^2}dx'\nonumber\\
&+C\|\nabla\varphi\|_{L^\infty}^2 \int_{|x'-x_{0}'|<s}\left(\frac{1}{\varepsilon+|x'|^2}+\frac{|x'-x_0'|^2}{(\varepsilon+|x'|^2)^2}\right)dx'+Cs^{n-1}\|\nabla^2\varphi\|_{L^\infty}^2.\label{ff}
\end{align}

{\bf Case 1.} For $0\leq |x_{0}'| \leq \sqrt{\varepsilon}$, $0<t<s<\sqrt{\varepsilon}$, from (\ref{w}) and (\ref{ff}), we have
\begin{align}
&\int_{\widehat{\Omega}_s(x_{0})}|w|^2dx\leq C\varepsilon^2\int_{\widehat{\Omega}_s(x_{0})}|\nabla w|^2dx,\label{w21}
\end{align}
and
\begin{align}
&\int_{\widehat{\Omega}_s(x_{0})}|\tilde{f}|^2dx\nonumber\\
&\leq C|\varphi(x_{0}',\varepsilon/2+h_{1}(x_{0}'))|^2\frac{s^{n-1}}{\varepsilon^2}+C\|\nabla\varphi\|_{L^\infty}^2 \frac{s^{n-1}}{\varepsilon}+Cs^{n-1}\|\nabla^2\varphi\|_{L^\infty}^2.\label{f21}
\end{align}
Denote $$F(t):=\int_{\widehat{\Omega}_t(x_{0})}|\nabla w|^2dx.$$ By (\ref{ww}), (\ref{w21}) and (\ref{f21}), for some universal constant $C_1>0$, we have for $0<t<s<\sqrt{\varepsilon}$,
\begin{align}\label{F}
F(t)\leq &\,\left(\frac{C_1\varepsilon}{s-t}\right)^2F(s)+C(s-t)^2s^{n-1} \cdot\nonumber\\
&\qquad\left(\frac{|\varphi(x_{0}',\varepsilon/2+h_{1}(x_{0}'))|^2}{\varepsilon^2}+\frac{\|\nabla\varphi\|_{L^\infty}^2}{\varepsilon}
+\|\nabla^2\varphi\|_{L^\infty}^2\right).
\end{align}
Let  $t_i=\delta+2C_1i\varepsilon$, $i=0, 1, \cdots$ and $k=\left[\frac{1}{4C_1\sqrt{\varepsilon}}\right]+1$, then $$\frac{C_1\varepsilon}{t_{i+1}-t_i}=\frac{1}{2}.$$
Using (\ref{F}) with $s=t_{i+1}$ and $t=t_i$, we obtain that, for $ i=0,1,2,\cdots, k,$
\begin{align*}
F(t_i)&\leq\frac{1}{4}F(t_{i+1})+C(i+1)^{n-1}\varepsilon^{n-1}\left(|\varphi(x_{0}',\varepsilon/2+h_{1}(x_{0}'))|^2+\varepsilon\|\varphi\|_{C^2(\Gamma_1^+)}^2\right).
\end{align*}
After $k$ iterations, making use of \eqref{lem2.2equ}, we have, for sufficiently small $\varepsilon$,
\begin{align*}
F(t_0)&\leq \big(\frac{1}{4}\big)^kF(t_k)+C\sum_{i=1}^k\big(\frac{1}{4}\big)^{i-1}i^{n-1}\varepsilon^{n-1}\left(|\varphi(x_{0}',\varepsilon/2+h_{1}(x_{0}'))|^2+\varepsilon\|\varphi\|_{C^2(\Gamma_1^+)}^2\right)\\
&\leq \big(\frac{1}{4}\big)^kF(\sqrt{\varepsilon})+C\varepsilon^{n-1}\left(|\varphi(x_{0}',\varepsilon/2+h_{1}(x_{0}'))|^2+\varepsilon\|\varphi\|_{C^2(\Gamma_1^+)}^2\right)\\
&\leq C\varepsilon^{n-1}\left[|\varphi(x_{0}',\varepsilon/2+h_{1}(x_{0}'))|^2+\varepsilon(\|\varphi\|_{C^2(\Gamma_{1}^{+})}^2+\|w\|^2_{L^2(\Omega_1)})\right],
\end{align*}
here we used that the first term in the last but one line decays exponentially, which implies that for $0\leq |x_{0}'| \leq \sqrt{\varepsilon}$,
\begin{align*}
\|\nabla w\|_{L^2(\widehat{\Omega}_\delta(x_0))}^{2}\leq  C\varepsilon^{n-1}\left[|\varphi(x_{0}',\varepsilon/2+h_{1}(x_{0}'))|^2+\varepsilon(\|\varphi\|_{C^2(\Gamma_{1}^{+})}^2+\|w\|^2_{L^2(\Omega_1)})\right].
\end{align*}

{\bf Case 2.} For $\sqrt{\varepsilon}\leq |x_{0}'|<\frac{1}{2}$, $0<t<s<\frac{2|x_{0}'|}{3}$, by (\ref{w}) and (\ref{ff}), we have
\begin{align*}
\int_{\widehat{\Omega}_s(x_{0})}|w|^2dx\leq C|x_{0}'|^4\int_{\widehat{\Omega}_s(x_{0})}|\nabla w|^2dx,
\end{align*}
and
\begin{align*}
&\int_{\widehat{\Omega}_s(x_{0})}|\tilde{f}|^2dx\\
& \leq C|\varphi(x_{0}',\varepsilon/2+h_{1}(x_{0}'))|^2\frac{s^{n-1}}{|x_{0}'|^4}+C\|\nabla\varphi\|_{L^\infty}^2 \frac{s^{n-1}}{|x_{0}'|^2}+Cs^{n-1}\|\nabla^2\varphi\|_{L^\infty}^2.
\end{align*}
Thus, we obtain that, for $0<t<s<\frac{2|x_{0}'|}{3}$,
 \begin{align}\label{F1}
F(t)\leq&\, \left(\frac{C_2|x_{0}'|^2}{s-t}\right)^2F(s)+C(s-t)^2s^{n-1}\cdot\nonumber\\
&\qquad\left(\frac{|\varphi(x_{0}',\varepsilon/2+h_{1}(x_{0}'))|^2}{|x_{0}'|^4}+\frac{\|\nabla\varphi\|_{L^\infty}^2}{|x_{0}'|^2}
+\|\nabla^2\varphi\|_{L^\infty}^2\right),
\end{align}
for some universal constant $C_2>0$. Taking the same iteration procedure as Case 1, set  $t_i=\delta+2C_2i|x_{0}'|^2$, $i=0, 1, \cdots$ and $k=\left[\frac{1}{4C_2|x_{0}'|}\right]+1$,
by (\ref{F1}) with $s=t_{i+1}$ and $t=t_i$, we have, for $i=0,1,2,\cdots, k$,
\begin{align*}
F(t_i)\leq&\,\frac{1}{4}F(t_{i+1})+C(i+1)^{n-1}|x_{0}'|^{2(n-1)}\left(|\varphi(x_{0}',\varepsilon/2+h_{1}(x_{0}'))|^2+|x_{0}'|^{2}\|\varphi\|_{C^2(\Gamma_{1}^{+})}^2\right).
\end{align*}
Similarly, after $k$ iterations, we have
\begin{align*}
F(t_0)\leq&\, \big(\frac{1}{4}\big)^kF(t_k)+C\sum_{i=1}^k\big(\frac{1}{4}\big)^{i-1}i^{n-1}|x_{0}'|^{2(n-1)}\cdot\\
&\left(|\varphi(x_{0}',\varepsilon/2+h_{1}(x_{0}'))|^2+|x_{0}'|^{2}\|\varphi\|_{C^2(\Gamma_{1}^{+})}^2\right)\\
\leq&\, \big(\frac{1}{4}\big)^kF(|x_0'|)+C|x_{0}'|^{2(n-1)}\left(|\varphi(x_{0}',\varepsilon/2+h_{1}(x_{0}'))|^2+|x_{0}'|^{2}\|\varphi\|_{C^2(\Gamma_{1}^{+})}^2\right)\\
\leq&\, C|x_{0}'|^{2(n-1)}\left[|\varphi(x_{0}',\varepsilon/2+h_{1}(x_{0}'))|^2+|x_{0}'|^{2}(\|\varphi\|_{C^2(\Gamma_{1}^{+})}^2+\|w\|^2_{L^2(\Omega_1)})\right],
\end{align*}
which implies that, for $ \sqrt{\varepsilon}\leq |x_{0}'|<\frac{1}{2}$,
\begin{align*}
&\|\nabla w\|_{L^2(\widehat{\Omega}_\delta(x_{0}))}^2\\
&\leq C|x_{0}'|^{2(n-1)}\left[|\varphi(x_{0}',\varepsilon/2+h_{1}(x_{0}'))|^2+|x_{0}'|^{2}(\|\varphi\|_{C^2(\Gamma_{1}^{+})}^2+\|w\|^2_{L^2(\Omega_1)})\right].
\end{align*}
The proof of Lemma \ref{lem2.2} is completed.
\end{proof}
\begin{lemma}\label{lem2.3}
For $l=1,\cdots,N$, if $|x'|\leq \sqrt{\varepsilon}$,
\begin{align}\label{equa2.9}
|\nabla w_l(x)|&\leq \frac{C|\varphi^{l}(x',\varepsilon/2+h_{1}(x'))-\psi^{l}(x',-\varepsilon/2+h_{2}(x'))|}{\sqrt{\varepsilon}}\nonumber\\
&\quad+C\left(\|\varphi^{l}\|_{C^2(\Gamma_{1}^{+})}+\|\psi^{l}\|_{C^2(\Gamma_{1}^{-})}+\|w_l\|_{L^2(\Omega_1)}\right),\end{align}
and if $\sqrt{\varepsilon}<|x'|<R_{0}$,
\begin{align}\label{equa2.10}|\nabla w_l(x)|\leq& \frac{C|\varphi^{l}(x',\varepsilon/2+h_{1}(x'))-\psi^{l}(x',-\varepsilon/2+h_{2}(x'))|}{|x'|}\nonumber\\
&+C\left(\|\varphi^{l}\|_{C^2(\Gamma_{1}^{+})}+\|\psi^{l}\|_{C^2(\Gamma_{1}^{-})}+\|w_l\|_{L^2(\Omega_1)}\right).\end{align}
Consequently, by (\ref{eq1.7}) and (\ref{eq1.7a}), we have for sufficiently small $\varepsilon$ and $x\in\Omega_{R_0}$,
\begin{align}\label{upper}
|\nabla v_l(x)|
\leq& \frac{C|\varphi^l(x',\varepsilon/2+h_{1}(x'))-\psi^{l}(x',-\varepsilon/2+h_{2}(x'))|}{\varepsilon+|x'|^2}\nonumber\\
&+C\left(\|\varphi^{l}\|_{C^2(\Gamma_{1}^{+})}+\|\psi^{l}\|_{C^2(\Gamma_{1}^{-})}+\|v_l\|_{L^2(\Omega_1)}\right).
\end{align}
 Moreover, if $\varphi^l(0', \frac{\varepsilon}{2})\neq\psi^l(0', -\frac{\varepsilon}{2})$,   then
\begin{align*}
|\nabla v_l(0',x_n)|\geq \frac{|\varphi^l(0', \frac{\varepsilon}{2})-\psi^l(0', -\frac{\varepsilon}{2})|}{C\varepsilon},\quad\forall\  x_n\in\left(-\frac{\varepsilon}{2},\frac{\varepsilon}{2}\right).
\end{align*}

\end{lemma}

\begin{proof}
Take the case when $\psi\equiv0$ and $l=1$ for instance, and denote $v:=v_1$, $w:=w_1$, $\tilde{u}:=\tilde{u}_{1}$ and $\varphi:=\varphi^1$. Given $x_0=(x_0',x_{0n})\in\Omega_{R_{0}}$, making a change of variables
\begin{align}\label{transform}
\begin{cases}
  x'-x_{0}'=\delta y', \\
x_{n}=\delta y_n.
\end{cases}
\end{align}
Define \begin{align*}
&\hat{h}_1(y'):=\frac{1}{\delta}\left(\frac{\varepsilon}{2}+h_1(\delta y'+x_{0}')\right),
\quad\hat{h}_2(y'):=\frac{1}{\delta}\left(-\frac{\varepsilon}{2}+h_2(\delta y'+x_{0}')\right).
\end{align*}
Then, the region $\widehat{\Omega}_\delta(x_0)$ becomes $Q_1$, where
$$Q_r=\{y\in \mathbb{R}^n\ |\ \hat{h}_2(y')<y_n<\hat{h}_1(y'),\ |y'|<r \},\quad 0<r\leq 1,$$ and the top and bottom boundaries of $Q_r$ become
$$\widehat{\Gamma}_r^+:=\{y\in \mathbb{R}^n\ |\ y_{n}=\hat{h}_1(y'),\ |y'|\leq r\}$$
and
$$\widehat{\Gamma}_r^-:=\{y\in \mathbb{R}^n\ |\ y_{n}=\hat{h}_2(y'),\ |y'|\leq r\},$$ respectively.
From (\ref{eq1.5'}) and the definition of $\hat{h}_1$ and $\hat{h}_2$, we have
\begin{align*}
\hat{h}_1(0')-\hat{h}_2(0')=1,\end{align*}and for $|y'|<1$,
\begin{align*}
&|\nabla \hat{h}_1(y')|+|\nabla \hat{h}_2(y')|\leq C(\delta+|x_{0}'|),\quad
|\nabla^2 \hat{h}_1(y')|+|\nabla^2 \hat{h}_2(y')|\leq C\delta.
\end{align*}
Since $R_{0}$ is small,  $\|\hat{h}_1\|_{C^{1, 1}((-1, 1))}$ and $\|\hat{h}_2\|_{C^{1, 1}((-1, 1))}$ are small and  $Q_1$ is essentially a unit square as far as applications of Sobolev embedding theorems and
classical $L^p$ estimates for elliptic systems are concerned.

Let \begin{align*}
\hat{u}(y', y_n)=\tilde{u}(\delta y'+x_{0}', \delta y_n),\qquad
\hat{w}(y', y_n)=w(\delta y'+x_{0}', \delta y_n).\end{align*}
 Thus, $\hat{w}(y)$ satisfies
 \begin{align*}
\begin{cases}
  \partial_\alpha\left(\hat{A}_{ij}^{\alpha\beta}\partial_\beta \hat{w}^j+\hat{B}_{ij}^\alpha \hat{w}^j\right)+\hat{C}_{ij}^\beta \partial_\beta \hat{w}^j+\hat{D}_{ij}\hat{w}^j=\hat{f}_{i},\quad
&\hbox{in}\  Q_1,  \\
\hat{w}=0, \quad&\hbox{on} \ \widehat{\Gamma}_1^\pm,
\end{cases}
\end{align*}
where
\begin{equation}\label{hatABCD}
\begin{aligned}
\hat{A}(y)=A(\delta y'+x_{0}', \delta y_n),\quad\ \ \hat{B}(y)=\delta\,B(\delta y'+x_{0}', \delta y_n), \\
\hat{C}(y)=\delta\,C(\delta y'+x_{0}', \delta y_n),\quad \hat{D}(y)=\delta^{2}D(\delta y'+x_{0}', \delta y_n),
\end{aligned}
\end{equation}
and $\hat{f}_{i}:=-\partial_\alpha\left(\hat{A}_{ij}^{\alpha\beta}\partial_\beta \hat{u}^j+\hat{B}_{ij}^\alpha \hat{u}^j\right)-\hat{C}_{ij}^\beta \partial_\beta \hat{u}^j-\hat{D}_{ij}\hat{u}^j$.

In view of $\hat{w}=0$ on the upper and lower boundaries of $Q_1$, we have, by Poincar\'{e} inequality, that
$$\|\hat{w}\|_{H^1(Q_1)}\leq C\|\nabla \hat{w}\|_{L^2(Q_1)}.$$
 Using the Sobolev embedding theorem and classical $W^{2, p}$ estimates for elliptic systems, 
  we have, for some $p>n$,
 $$\|\nabla \hat{w}\|_{L^\infty(Q_{1/2})}\leq C\|\hat{w}\|_{W^{2, p}(Q_{1/2})}\leq C\left(\|\nabla \hat{w}\|_{L^2(Q_1)}+\|\hat{f}\|_{L^\infty(Q_1)}\right).$$
Since $$\|\nabla \hat{w}\|_{L^\infty(Q_{1/2})}=\delta\|\nabla {w}\|_{L^\infty(\widehat{\Omega}_{\delta/2}(x_0))},\quad\|\nabla\hat{w}\|_{L^2(Q_1)}=\delta^{1-\frac{n}{2}}\|\nabla w\|_{L^2(\widehat{\Omega}_\delta(x_0))}$$and$$\|\hat{f}\|_{L^\infty(Q_1)}=\delta^2\|\tilde{f}\|_{L^\infty(\widehat{\Omega}_\delta(x_0))},$$
tracing back to $w$ through the transforms, we have
\begin{align*}
\|\nabla w\|_{L^\infty(\widehat{\Omega}_{\delta/2}(x_0))}\leq\frac{C}{\delta}\left(\delta^{1-\frac{n}{2}}\|\nabla w\|_{L^2(\widehat{\Omega}_\delta(x_0))}+\delta^2\|\tilde{f}\|_{L^\infty(\widehat{\Omega}_\delta(x_0))}\right).
\end{align*}
{\bf Case 1.} For $0\leq |x_{0}'|\leq \sqrt{\varepsilon}$.

By  (\ref{f}) and (\ref{step2}), we have
\begin{align*}
&\delta^{-\frac{n}{2}}\|\nabla w\|_{L^2(\widehat{\Omega}_\delta(x_0))}\\
&\leq \frac{C}{\sqrt{\varepsilon}}\left(\frac{\varepsilon}{\delta}\right)^{\frac{n}{2}}\left[|\varphi(x_{0}',\varepsilon/2+h_{1}(x_{0}'))|+\sqrt{\varepsilon}(\|\varphi\|_{C^2(\Gamma_{1}^{+})}+\|w\|_{L^2(\Omega_1)})\right]\\
&\leq \frac{C}{\sqrt{\varepsilon}}|\varphi(x_{0}',\varepsilon/2+h_{1}(x_{0}'))|
+C(\|\varphi\|_{C^2(\Gamma_{1}^{+})}+\|w\|_{L^2(\Omega_1)})\end{align*}
and
\begin{align*}
\delta \|\tilde{f}\|_{L^\infty(\widehat{\Omega}_\delta(x_0))}\leq \frac{C}{\sqrt{\varepsilon}}|\varphi(x_{0}',\varepsilon/2+h_{1}(x_{0}'))|
+C(\|\nabla\varphi\|_{L^\infty}+\|\nabla^2\varphi\|_{L^\infty}).
\end{align*}
Therefore,
\begin{align*}\|\nabla w\|_{L^\infty(\widehat{\Omega}_{\delta/2}(x_0))}\leq \frac{C}{\sqrt{\varepsilon}}|\varphi(x_{0}',\varepsilon/2+h_{1}(x_{0}'))|
+C(\|\varphi\|_{C^2(\Gamma_{1}^{+})}+\|w\|_{L^2(\Omega_1)}).
\end{align*}
\eqref{equa2.9} is proved.

{\bf Case 2.} For $\sqrt{\varepsilon}\leq |x_{0}'|\leq R_{0}$.

Using  (\ref{f}) and (\ref{Step2}), we obtain
\begin{align*}
&\delta^{-\frac{n}{2}}\|\nabla w\|_{L^2(\widehat{\Omega}_\delta(x_0))}\\
&\leq \frac{C}{|x_{0}'|}\left(\frac{|x_{0}'|^2}{\delta}\right)^{\frac{n}{2}}\left[|\varphi(x_{0}',\varepsilon/2+h_{1}(x_{0}'))|+|x_{0}'|(\|\varphi\|_{C^2(\Gamma_{1}^{+})}+\|w\|_{L^2(\Omega_1)})\right]\\
&\leq \frac{C}{|x_{0}'|}|\varphi(x_{0}',\varepsilon/2+h_{1}(x_{0}'))|
+C(\|\varphi\|_{C^2(\Gamma_{1}^{+})}+\|w\|_{L^2(\Omega_1)}),\end{align*}
and
\begin{align*}
\delta \|\tilde{f}\|_{L^\infty(\widehat{\Omega}_\delta(x_0))}\leq \frac{C}{|x_{0}'|}|\varphi(x_{0}',\varepsilon/2+h_{1}(x_{0}'))|
+C(\|\nabla\varphi\|_{L^\infty}+\|\nabla^2\varphi\|_{L^\infty}).
\end{align*}
Therefore,
\begin{align*}\|\nabla w\|_{L^\infty(\widehat{\Omega}_{\delta/2}(x_0))}\leq \frac{C}{|x_{0}'|}|\varphi(x_{0}',\varepsilon/2+h_{1}(x_{0}'))|
+C(\|\varphi\|_{C^2(\Gamma_{1}^{+})}+\|w\|_{L^2(\Omega_1)}).\end{align*}
\eqref{equa2.10} is proved.

Notice that $|\nabla v|\leq|\nabla w|+|\nabla \tilde{u}|$,   by (\ref{eq1.7}), (\ref{eq1.7a}), \eqref{equa2.9} and \eqref{equa2.10}, we obtain (\ref{upper}).
By the Taylor expansion and (\ref{h0}), we have
\begin{align}\label{taylor}
\varphi^l(x', \frac{\varepsilon}{2}+h_1(x'))
=&\varphi^l(0', \frac{\varepsilon}{2})+\nabla_{x'}\varphi^l(0', \frac{\varepsilon}{2})\cdot x'+O(|x'|^2).
\end{align}
It is clear that if $\varphi^l(0', \frac{\varepsilon}{2})\neq0$,  then
\begin{align*}
|\nabla v_l(0',x_n)|\geq \frac{|\varphi^l(0', \frac{\varepsilon}{2})|}{C\varepsilon},\quad\forall\  x_n\in\left(-\frac{\varepsilon}{2},\frac{\varepsilon}{2}\right).
\end{align*}
The proof of Lemma \ref{lem2.3} is finished.
\end{proof}

\begin{proof}[{\bf Proof of Theorem \ref{thm1.1}}]
By Lemma 2.1--Lemma 2.3, we have, for $x\in\Omega_{R_0}$,
\begin{align*}
|\nabla u(x)|&\leq\sum_{l=1}^{N}|\nabla{v}_{l}|\\
&\leq \frac{C|\varphi^l(x',\varepsilon/2+h_{1}(x'))-\psi^l(x',-\varepsilon/2+h_{2}(x'))|}{\varepsilon+|x'|^2}\nonumber\\
&+C\left(\|\varphi\|_{C^2(\Gamma_{1}^{+})}
+\|\psi\|_{C^2(\Gamma_{1}^{-})}+\|u\|_{L^2(\Omega_1)}\right).
\end{align*}
 Applying the standard   elliptic theorem (see Agmon et al. \cite{AD1} and \cite{AD2}), we have
 $$\|\nabla u\|_{L^\infty(\Omega_{1/2}\setminus \Omega_{R_{0}})}\leq C\left(\|\varphi\|_{C^2(\Gamma_1^{+})}+\|\psi\|_{C^2(\Gamma_1^{-})}+\|u\|_{L^2(\Omega_{1})}\right).$$
If $\varphi^l(0', \frac{\varepsilon}{2})\neq\psi^l(0', -\frac{\varepsilon}{2})$   for some integer $l$, then by Lemma 2.3,  we obtain
$$|\nabla u(0',x_n)|\geq\frac{|\varphi^l(0', \frac{\varepsilon}{2})-\psi^l(0', -\frac{\varepsilon}{2})|}{C\varepsilon},\quad\forall\ x_n\in\left(-\frac{\varepsilon}{2}, \frac{\varepsilon}{2}\right).$$
 The proof of Theorem \ref{thm1.1} is completed.
\end{proof}
\vspace{.5cm}
\section{ Proof of Corollary \ref{corol1.2}}\label{sec_3}
Since the problem (\ref{eq1.1'}) has much more applications in practice, such as conductivity problem, anti-plane shear model. We give a sketched proof of Corollary \ref{corol1.2} and only list its main ingredients.
We use the auxillary scalar function $\bar{u}\in C^2(\mathbb{R}^n)$ introduced in Section 2, which satisfying (\ref{bar_u}).
Define
\begin{align*}
\tilde{u}(x)=\varphi(x',\frac{\varepsilon}{2}+h_1(x'))\bar{u}+\psi(x',-\frac{\varepsilon}{2}+h_2(x'))(1-\bar{u}),
\end{align*} then
\begin{align}\label{eq1.7'}
|\nabla_{x'}\tilde{u}(x)|&\leq\frac{C|x'|}{\varepsilon+|x'|^2}|\varphi(x',\varepsilon/2+h_{1}(x'))-\psi(x',-\varepsilon/2+h_{2}(x'))|\nonumber\\
&\quad+C(\|\nabla\varphi\|_{L^{\infty}}+\|\nabla\psi\|_{L^{\infty}}),
\end{align}
\begin{align}\label{eq1.7a'}
&\frac{|\varphi(x',\varepsilon/2+h_{1}(x'))-\psi(x',-\varepsilon/2+h_{2}(x'))|}{C(\varepsilon+|x'|^2)}\leq\nonumber\\
&|\partial_{n}\tilde{u}(x)|\leq\frac{C|\varphi(x',\varepsilon/2+h_{1}(x'))
-\psi(x',-\varepsilon/2+h_{2}(x'))|}{\varepsilon+|x'|^2},\quad\quad
\end{align}
and
\begin{align}
|\partial_{\alpha \alpha}\tilde{u}(x)|&\leq\frac{C}{\varepsilon+|x'|^2}|\varphi(x',\varepsilon/2+h_{1}(x'))-\psi(x',-\varepsilon/2+h_{2}(x'))|\nonumber\\
&\quad+C\left(\frac{|x'|}{\varepsilon+|x'|^2}+1\right)(\|\nabla\varphi\|_{L^{\infty}}+
\|\nabla\psi\|_{L^{\infty}})\nonumber\\
&\quad+C(\|\nabla^{2}\varphi\|_{L^{\infty}}+\|\nabla^{2}\psi\|_{L^{\infty}}),\quad \alpha=1,\cdots, n-1,\label{eq1.8'}\\
\partial_{n n}\tilde{u}(x)&=0.\label{eq1.10'}
\end{align}

 Denote $w=u-\tilde{u}$, which
satisfies the following boundary value problem
\begin{align}\label{eq2.35}
\begin{cases}
  \Delta w=-\Delta\tilde{u},
&\hbox{in}\  \Omega_1,  \\
w=0,\ &\hbox{on}\ \Gamma_1^{\pm}.
\end{cases}
\end{align}

{\bf Step 1.} Boundedness of $\int_{\Omega_{1/2}}|\nabla w|^2dx$.

 Multiplying the equation in (\ref{eq2.35})
by $w$ and applying integration by parts on $\Omega_{1/2}$, in view of $w=0$ on $\Gamma_{1}^{\pm}$, we have
\begin{align*}
&\int_{\Omega_{1/2}}|\nabla w|^2dx\\
&=\int_{\Omega_{1/2}}\sum_{\alpha=1}^{n-1}w\partial_{\alpha\alpha} \tilde{u} dx+\int\limits_{\scriptstyle |x'|={\frac{1}{2}},\atop\scriptstyle
-\frac{\varepsilon}{2}+h_2(x')<x_{n}<\frac{\varepsilon}{2}+h_1(x')\hfill}\sum_{\alpha=1}^{n-1} w\partial_\alpha w  \frac{x_\alpha}{r}ds\\
&=-\int_{\Omega_{1/2}}\nabla_{x'}w\cdot \nabla_{x'}\tilde{u} dx+\int\limits_{\scriptstyle |x'|={\frac{1}{2}},\atop\scriptstyle
-\frac{\varepsilon}{2}+h_2(x')<x_{n}<\frac{\varepsilon}{2}+h_1(x')\hfill}
\sum_{\alpha=1}^{n-1}w(\partial_{\alpha}\tilde{u}+\partial_\alpha w) \frac{x_\alpha}{r}ds\\
&\leq \frac{1}{2} \int_{\Omega_{1/2}}|\nabla w|^2dx+\frac{1}{2} \int_{\Omega_{1/2}}|\nabla_{x'}\tilde{u}|^2dx\\
&\quad+\int\limits_{\scriptstyle |x'|={\frac{1}{2}},\atop\scriptstyle
-\frac{\varepsilon}{2}+h_2(x')<x_{n}<\frac{\varepsilon}{2}+h_1(x')\hfill}C(|\nabla_{x'}\tilde{u}|^2+|w|^2+|\nabla w|^2)ds.
\end{align*}
By using (\ref{eq1.7'}),
\begin{align*}
\int_{\Omega_{1/2}}|\nabla_{x'}\tilde{u}|^2dx\leq C(\|\varphi\|^2_{C^1(\Gamma_1^+)}+\|\psi\|^2_{C^1(\Gamma_1^-)}),\quad\quad\quad\quad\quad\quad\quad\\
\int\limits_{\scriptstyle |x'|={\frac{1}{2}},\atop\scriptstyle
-\frac{\varepsilon}{2}+h_2(x')<x_{n}<\frac{\varepsilon}{2}+h_1(x')\hfill}|\nabla_{x'}\tilde{u}|^2ds
\leq C(\|\varphi\|^2_{C^1(\Gamma_1^+)}+\|\psi\|^2_{C^1(\Gamma_1^-)}). \quad\quad\quad\quad
\end{align*}
Using (\ref{eq1.8'}) and (\ref{eq1.10'}), for $x\in\Omega_{1}\setminus \overline{\Omega_{1/4}}$, $$
|\Delta\tilde{u}(x)|\leq C(\|\varphi\|_{C^2(\Gamma_1^+)}+\|\psi\|_{C^2(\Gamma_1^-)}).$$
In view of $w=0$ on $\Gamma_1^{\pm}$, by a standard elliptic theorem  for Poisson equation, we have
\begin{align*}
\|w\|_{L^\infty(\Omega_{2/3}\setminus \overline{\Omega_{1/3}})}^2+\|\nabla w\|_{L^\infty(\Omega_{2/3}\setminus \overline{\Omega_{1/3}})}^2
&\leq C(\|w\|_{L^2(\Omega_1\setminus \overline{\Omega_{1/4}})}+\|\Delta\tilde{u}\|_{L^\infty(\Omega_1\setminus \overline{\Omega_{1/4}})}) \\ &\leq C(\|w\|^2_{L^2(\Omega_1)}+\|\varphi\|^2_{C^2(\Gamma_1^+)}+\|\psi\|^2_{C^2(\Gamma_1^-)}),
\end{align*}
which, implies that
\begin{align}\label{lem2.2equ'}
\int_{\Omega_{1/2}}|\nabla w|^2dx\leq C(\|w\|^2_{L^2(\Omega_1)}+\|\varphi\|^2_{C^2(\Gamma_1^+)}+\|\psi\|^2_{C^2(\Gamma_1^-)}),
\end{align}
where $C$ depends on $n$, $\lambda$, $\kappa_0$ and $\kappa_1$.

{\bf Step 2.} Estimate   of $\int_{\widehat{\Omega}_\delta(x_{0})}|\nabla w|^2dx$, where $x_{0}\in\Omega_{1/2},~\delta=\delta(x_0'):=\varepsilon+h_1(x_0')-h_2(x_0')$.


Let $\eta(x')$ be the  cut-off function in Lemma 2.2.
Multiplying $\eta^2w$ on both sides of the equation in (\ref{eq2.35}) and applying integration by parts, we have
\begin{align*}
\int_{\widehat{\Omega}_s(x_{0})}\nabla w\cdot\nabla(\eta^2w)dx
=\int_{\widehat{\Omega}_s(x_{0})}\Delta \tilde{u}\eta^2wdx.
\end{align*}
Note that
\begin{align*}
\int_{\widehat{\Omega}_s(x_{0})}\nabla w\cdot\nabla(\eta^2w)dx
=\int_{\widehat{\Omega}_s(x_{0})}|\nabla(\eta w)|^{2}dx-\int_{\widehat{\Omega}_s(x_{0})}w^{2}|\nabla\eta|^{2}dx,
\end{align*}
by the Cauchy inequality, we have
\begin{align}\label{2.39}
\int_{\widehat{\Omega}_t(x_{0})}|\nabla w|^2dx&\leq\int_{\widehat{\Omega}_s(x_{0})}|\nabla(\eta w)|^2dx\nonumber\\
&=\int_{\widehat{\Omega}_s(x_{0})}w^{2}|\nabla\eta|^{2}dx
+\int_{\widehat{\Omega}_s(x_{0})}\Delta \tilde{u}\eta^2wdx\nonumber\\
&\leq\frac{C}{(s-t)^2}\int_{\widehat{\Omega}_s(x_{0})}|w|^2dx+(s-t)^2\int_{\widehat{\Omega}_s(x_{0})}|\Delta \tilde{u}|^2dx.
\end{align}
It is easy to see that \begin{align*}
\int_{\widehat{\Omega}_s(x_{0})}|w|^2dx
&\leq\int_{|x'-x_{0}'|<s}C(\varepsilon+|x'|^2)^2\int_{-\frac{\varepsilon}{2}+h_2(x')}^{\frac{\varepsilon}{2}+h_1(x')}|\nabla w|^2dx_{n}dx',
\end{align*} and by (\ref{eq1.8'})-(\ref{eq1.10'}),
\begin{align*}
&\int_{\widehat{\Omega}_s(x_{0})}|\Delta\tilde{u}|^2dx\nonumber\\
&\leq C|\varphi(x_0',\varepsilon/2+h_1(x_0'))-\psi(x_0',-\varepsilon/2+h_{2}(x_0'))|^2\int_{|x'-x_0'|<s}\frac{1}{\varepsilon+|x'|^2}dx'\nonumber\\
&\quad+C(\|\nabla\varphi\|_{L^\infty}^2+\|\nabla\psi\|_{L^\infty}^2)\int_{|x'-x_0'|<s}\left(\frac{|x-x_0'|^2}{\varepsilon+|x'|^2}+
1\right)dx'\nonumber\\
&\quad+Cs^{n-1}(\|\nabla^2\varphi\|_{L^\infty}^2+\|\nabla^2\psi\|_{L^\infty}^2).
\end{align*}

{\bf Case 1.} For $|x_{0}'|\leq\sqrt{\varepsilon}$, $0<t<s<\sqrt{\varepsilon}$,   we have
\begin{align}\label{2.40}
\int_{\widehat{\Omega}_s(x_{0})}|w|^2dx
\leq C {\varepsilon}^2\int_{\widehat{\Omega}_s(x_{0})}|\nabla w|^2dx,
\end{align}
and
\begin{align}\label{2.41}
\int_{\widehat{\Omega}_s(x_{0})}|\Delta\tilde{u}|^2dx&\leq \frac{Cs^{n-1}}{\varepsilon}|\varphi(x_0',\varepsilon/2+h_1(x_0'))-\psi(x_0',-\varepsilon/2+h_{2}(x_0'))|^2\nonumber\\
&\quad+Cs^{n-1}(\|\varphi\|_{C^2(\Gamma_1^+)}^2+\|\psi\|_{C^2(\Gamma_1^-)}^2).
\end{align}
Denote $F(t):=\int_{\widehat{\Omega}_t(x_{0})}|\nabla w|^2dx$. By (\ref{2.39})-(\ref{2.41}), for some universal constant $\hat{C}_1>0$, we have for $0<t<s<\sqrt{\varepsilon}$,
\begin{align}\label{F'}
F(t)&\leq\left(\frac{\hat{C}_1\varepsilon}{s-t}\right)^2F(s)\nonumber\\
&\quad+C(s-t)^2s^{n-1}(\frac{1}{\varepsilon}|\varphi(x_0',\varepsilon/2+h_1(x_0'))
-\psi(x_0',-\varepsilon/2+h_{2}(x_0'))|^2
\nonumber\\
&\quad+\|\varphi\|_{C^2(\Gamma_1^+)}^2+\|\psi\|_{C^2(\Gamma_1^-)}^2).
\end{align}
Let $t_i=\delta+2i\hat{C}_1\varepsilon,\ i=0,1,\cdots$ and $k=\left[\frac{1}{4\hat{C}_1\sqrt{\varepsilon}}\right]+1$, then
$$\frac{\hat{C}_1\varepsilon}{t_{i+1}-t_i}=\frac{1}{2}.$$
Using (\ref{F'}) with $s=t_{i+1}$ and $t=t_i$, we obtain that, for $ i=0,1,2,\cdots, k,$
\begin{align*}
F(t_i)&\leq\frac{1}{4}F(t_{i+1})\\
&+C(i+1)^{n-1}\varepsilon^{n}
[|\varphi(x_{0}',\varepsilon/2+h_{1}(x_{0}'))-\psi(x_0',-\varepsilon/2+h_{2}(x_0'))|^2\\
&+\varepsilon(\|\varphi\|_{C^2(\Gamma_1^+)}^2+\|\psi\|_{C^2(\Gamma_1^-)}^2)].
\end{align*}
After $k$ iterations, making use of \eqref{lem2.2equ'}, we have, for sufficiently small $\varepsilon$,
\begin{align*}
F(t_0)
&\leq C\varepsilon^{n}[|\varphi(x_{0}',\varepsilon/2+h_{1}(x_{0}'))-\psi(x_0',-\varepsilon/2+h_{2}(x_0'))|^2\\
&\quad+\varepsilon(\|w\|^2_{L^2(\Omega_1)}+\|\varphi\|_{C^2(\Gamma_{1}^{+})}^2+\|\psi\|_{C^2(\Gamma_{1}^{-})}^2)],
\end{align*}
here we used that the first term in the last but one line decays exponentially, which implies that for $0\leq |x_{0}'| \leq \sqrt{\varepsilon}$,
\begin{align}\label{2.42}
\|\nabla w\|_{L^2(\widehat{\Omega}_\delta(x_0))}^{2}&\leq  C\varepsilon^{n} [|\varphi(x_{0}',\varepsilon/2+h_{1}(x_{0}'))-\psi(x_0',-\varepsilon/2+h_{2}(x_0'))|^2\nonumber\\
&\quad +\varepsilon(\|w\|^2_{L^2(\Omega_1)}+\|\varphi\|_{C^2(\Gamma_{1}^{+})}^2+\|\psi\|_{C^2(\Gamma_{1}^{-})}^2)].
\end{align}

{\bf Case 2.} Similarly, for $\sqrt{\varepsilon}<|x_0'|<\frac{1}{2}$, $0<t<s<\frac{2|x_0'|}{3}$,  we have
\begin{align}\label{2.43}
\|\nabla w\|_{L^2(\widehat{\Omega}_\delta(x_0))}^{2}&\leq  C{|x_0'|}^{2n}[|\varphi(x_{0}',\varepsilon/2+h_{1}(x_{0}'))-\psi(x_0',-\varepsilon/2+h_{2}(x_0'))|^2\nonumber\\
&\quad+|x_0'|^2(\|w\|^2_{L^2(\Omega_1)}+\|\varphi\|_{C^2(\Gamma_{1}^{+})}^2+\|\psi\|_{C^2(\Gamma_{1}^{-})}^2)].
\end{align}

{\bf Step 3.} Estimate of $|\nabla w(x)|$ for $x\in\Omega_{R_0}$, for some small $R_0\in(0,\frac{1}{2})$.

By using a similar argument in Lemma 2.3,
it follows from (\ref{2.42}) and (\ref{2.43})  that, for $0\leq |x_{0}'|\leq \sqrt{\varepsilon}$,
\begin{align*}\|\nabla w\|_{L^\infty(\widehat{\Omega}_{\delta/2}(x_0))}&\leq C|\varphi(x_{0}',\varepsilon/2+h_{1}(x_{0}'))-\psi(x_0',-\varepsilon/2+h_2(x_0'))|\\
&\quad+C\sqrt{\varepsilon}(\|w\|_{L^2(\Omega_1)}+\|\varphi\|_{C^2(\Gamma_{1}^{+})}+\|\psi\|_{C^2(\Gamma_{1}^{-})}).
\end{align*}
and for
$\sqrt{\varepsilon}\leq |x_{0}'|\leq R_{0}$,
\begin{align*}\|\nabla w\|_{L^\infty(\widehat{\Omega}_{\delta/2}(x_0))}&\leq C|\varphi(x_{0}',\varepsilon/2+h_{1}(x_{0}'))-\psi(x_0',-\varepsilon/2+h_2(x_0'))|\\
&\quad+C|x_{0}'|(\|w\|_{L^2(\Omega_1)}+\|\varphi\|_{C^2(\Gamma_{1}^{+})}+\|\psi\|_{C^2(\Gamma_{1}^{-})}).\end{align*}
 By using (\ref{eq1.7'})-(\ref{eq1.7a'}), we obtain (\ref{mainestimate'}).

Finally, if $\varphi(0', \frac{\varepsilon}{2})\neq\psi(0', -\frac{\varepsilon}{2})$, then by the Taylor expansion, we obtain
$$|\nabla u(0', x_n)|\geq\frac{|\varphi(0', \frac{\varepsilon}{2})-\psi(0', -\frac{\varepsilon}{2})|}{C\varepsilon},\quad\forall\ x_n\in\left(-\frac{\varepsilon}{2}, \frac{\varepsilon}{2}\right).$$
The proof of  Corollary 1.3 is completed.


\bibliographystyle{plain}

\def\cprime{$'$}

\end{document}